\documentclass[a4paper]{article}

\usepackage[letterpaper,left=1in,right=1in,top=0.8in,bottom=0.8in]{geometry}

\usepackage{amssymb}

\usepackage[all]{xy}
\usepackage{ucs}   
\usepackage[utf8x]{inputenc}  
\usepackage[english]{babel}
\usepackage{graphicx} 

\usepackage[export]{adjustbox}
\usepackage{enumerate}
\usepackage{amsthm}
\usepackage{amsmath}
\usepackage{mathtools}
\usepackage{tikz-cd}
\usepackage{bbm}
\usepackage{enumitem}

\usepackage{verbatim}
\usepackage{parskip}

\theoremstyle{plain}
\newtheorem{Th}{Theorem}[section]

\newtheorem{Cor}[Th]{Corollary}
\newtheorem{Prop}[Th]{Proposition}

\theoremstyle{definition}
\newtheorem{Def}[Th]{Definition}

\newtheorem{Rem}[Th]{Remark}

\newtheorem{?}[Th]{Problem}
\newtheorem{Ex}[Th]{Example}

\newtheorem{Ass}[Th]{Assumption}
\newtheorem{Proper}[Th]{Property}
\newtheorem{Example}[Th]{Example}

\newenvironment{sketch} {\begin{proof}[Sketch of proof]}
	{\end{proof}}

\def\<{\langle}
\def\>{\rangle}

\newcommand{\Aa}{{\mathbb A}}

\newcommand{\Spec}{\operatorname{Spec}}

\newcommand{\Sym}{\operatorname{Sym}}

\newcommand{\End}{\operatorname{End}}
\newcommand{\GW}{\operatorname{GW}}
\newcommand{\SH}{\operatorname{SH}}
\newcommand{\DM}{\operatorname{DM}}
\newcommand{\one}{\mathbbm{1}}
\newcommand{\OO}{\mathcal{O}}
\newcommand{\Sp}{\operatorname{sp}}
\newcommand{\chic}{\operatorname{\chi_c}}
\renewcommand{\P}{{\mathbb{P}}}
\newcommand{\Tr}{\operatorname{Tr}}
\newcommand{\Dim}{\operatorname{dim}}
\newcommand{\C}{\mathbb{C}}
\newcommand{\Z}{\mathbb{Z}}
\newcommand{\G}{\mathbb{G}}
\newcommand{\Char}{\operatorname{char}}
\newcommand{\tr}{\operatorname{tr}}

\newcommand{\Sch}{\operatorname{\mathbf{Sch}}}

\newcommand{\id}{\operatorname{id}}

\newcommand{\rk}{\operatorname{rk}}
\newcommand{\Bl}{\operatorname{Bl}}
\newcommand{\Var}{\operatorname{Var}}
\newcommand{\var}{\operatorname{var}}
\newcommand{\MF}{\operatorname{MF}}
\newcommand{\Gal}{\operatorname{Gal}}

\numberwithin{equation}{section}
\usepackage[colorlinks=false,plainpages,urlcolor=blue,linktocpage=true, hidelinks]{hyperref}

\theoremstyle{plain}

\title{{Motivic nearby cycles and their monodromy at a singular point}}
\author{Ran Azouri}
\date{}

\begin{document}
	\maketitle
	
	\begin{abstract}
		In this survey article\footnote{Based on the works of Marc Levine, Simon Pepin Lehalleur and Vasudevan Srinivas \cite{Le20b}; the author \cite{Az}; and Emil Jacobsen and the author \cite{AJ}.}, we explain how to compute both the quadratic Euler characteristic of nearby cycles, and the motivic monodromy, at a quasi-homogeneous singularity. This gives, for such singularity, a quadratic refinement to the Deligne--Milnor formula in characteristic zero, and an enhancement of the Picard--Lefschetz formula to Voevodsky motives with rational coefficients.
	\end{abstract}

 \setcounter{tocdepth}{1}
\tableofcontents
 \paragraph{Acknowledgements}
  The author would like to thank Marc Levine for the invitation to write this survey for the IAS–PCMI 2024 conference volume, and for his guidance on the quadratic Milnor formula problem. The author also thanks Emil Jacobsen for collaboration on the Picard–Lefschetz problem, and Joseph Ayoub for his guidance and valuable discussions. The author is grateful to the Park City Mathematics Institute and the summer session 2024 organizers for their hospitality, and for the invitation to give a talk in the research program on which these notes are based.  The author was supported by
  the project Foundations of Motivic Real K-Theory
  (ERC grant no.\ 949583).

	\section{The Milnor fiber and nearby cycles}
	
	\subsection{In topology - Milnor} \label{sec:topology}
	
	The story begins with John Milnor's work \emph{Singular points of complex hypersurfaces} (1968). Let $f \colon \C^{n+1} \to \C $ be a holomorphic function, smooth except at one point $p \in X_0 := f^{-1} (0)$. For sufficiently small $\varepsilon > 0$ and even smaller $0 < |t| < \varepsilon$, the \emph{Milnor fiber at $p$} is defined as
		\[
	\MF_{f,p} := X_t \cap B_p(\varepsilon),
	\]
	where $B_p(\varepsilon) \subset \C^{n+1} $ is the ball of radius $\varepsilon$ around $p$, and $X_t : = f^{-1} (t)$.  

	Milnor showed that $\MF_{f,p}$ is homotopy equivalent to a wedge of spheres,
	\[
	\MF_{f,p} \simeq \bigvee_{1}^{\mu} S^n .
	\]
	The integer $\mu$ is the \emph{Milnor number}, given also by
	$
	\mu = \dim H^n(\MF_{f,p}, \Z)
	$.
	Looking for invariants related to an isolated singular point we have the following facts:
	\begin{itemize}
		\item[(a)] \textbf{The Milnor formula:}
	The Milnor number satisfies the formula 
	\begin{equation} \label{eq:Milnor}
		\chi(\MF_{f,p}) - 1 = (-1)^n \mu.
	\end{equation}
	In addition, the Milnor number can be computed algebraically as the dimension of the Jacobian ring
	\[\mu  = \dim J(f,p) , \] where  $J(f,p) := \mathbb{C}[x_1,...,x_{n+1}]_p/((\partial f/ \partial x_i)_i)$. 
	Milnor and Orlik also computed $\mu$ for a singularity defined by a weighted homogeneous polynomial, \cite{MO}.
			\item[(b)] \textbf{The monodromy:}
			 Rotating around $p$ defines the monodromy action on cohomology $T \circlearrowright  H^n(\MF, \mathbb{Z})$. Picard \cite{Pic}, and Lefschetz \cite{Lef}, computed this explicitly, although they did not have the notion of Milnor fiber available at the time.
	\end{itemize}

	\subsection{In algebraic geometry -- Deligne}
	 In the algebro--geometric setting, we start with stating the assumptions that we use throughout the article.
	\begin{Ass} \label{assumption}
		Let $k$ ba a base field of characteristic different than $2$, let $X$ be a regular scheme, and let $f \colon X \to S$ be a flat morphism of finite type.
		
		Our base $S$ will be either the affine line, $S=\Aa^1_k= \Spec k[t]$, or a trait, i.e., $S = \Spec \OO$, where $\OO$ is a discrete valuation ring, essentially of finite type over $k$, and $t \in \OO$ is a uniformizer; in both cases we have a parameter $t$. Note that can move from the former to the latter base by localizing at $0$, and from the latter to the former by composing with $t \colon \Spec \OO \to \Spec k[t]$. We can restrict $f$ to the generic fiber, $X_\eta \to \eta$, with $\eta$ being, respectively, either $\mathbb{G}_m$ or the fraction field; and to the special fiber, $X_\sigma \to \sigma$, with $\sigma$ being respectively $0$, or the residue field,
		\begin{equation*}
			\begin{tikzcd}
				X_{\sigma} \arrow[r, hook, "\bar{i}"] \arrow[d, "f_\sigma" ']
				& X  \arrow[d, "f"]
				& X_{\eta} \arrow[l, hook', "\bar{j}" '] \arrow[d, "f_{\eta}"]
				\\
				\sigma \arrow[r, hook]
				& S
				& \eta \arrow[l, hook']
			\end{tikzcd} .
		\end{equation*}
	\end{Ass}
	
	Pierre Deligne introduced the \emph{nearby cycles functor} in \cite[Expos\'e XIII]{SGA7II} after Grothendieck,
	\[
	\Psi_f^{\acute{e}t} : D^b_c(X_\eta) \to D^b_c(X_\sigma).
	\]
A related functor is the \emph{vanishing cycles}, $\Phi_f^{\acute{e}t}$, which measures the failure of $\Psi_f^{\acute{e}t}$ to be the restriction of the constant sheaf. We have the \emph{nearby cycles sheaf},
	$
	\Psi_f^{\acute{e}t} \Z_\ell $.
		Locally at a point, for $p \in X_s$, the stalk of $\Psi_f^{\acute{e}t} \Z_\ell$ at $p$ computes the $\ell$-adic cohomology of the Milnor fiber, where here we set $\MF_{f,p} : = (X_{\bar{p}})_{\bar{\eta}}$ (the generic fiber of the strict henselization of $X$ at the geometric point $\bar{p}$):
		\[
	(R \Psi_f^{\acute{e}t} \Z_\ell)_{\bar{p}} \simeq R \Gamma (\MF_{f,p}, \Z_\ell). 
		\]
Globally, if  $f$ is proper, then the global sections of $\Psi_f^{\acute{e}t} \Z_\ell$ recover the cohomology of the generic fiber:
		\begin{equation} \label{eq:properetale}
		R\Gamma(X_{\bar{\sigma}}, \Psi_f^{\acute{e}t} \Z_\ell) \cong R\Gamma(X_{\bar{\eta}}, \Z_\ell).
		\end{equation}
		We may think of the nearby cycles sheaf then, as the analogue of the Milnor fiber in algebraic geometry.
	We have the following important results, which are algebraic geometric versions for $(a)$ and $(b)$ in $\S~\ref{sec:topology}$.
	\begin{itemize}
		\item[(a)] \textbf{The Deligne-Milnor conjecture} \cite[Expos\'e XVI - \emph{La formule de Milnor}]{SGA7II}. We have the formula at a point $p$ \cite[Conjecture~1.9]{SGA7II}, 
		\begin{equation} \label{eq:DeligneFormula} \chi ^{\ell-adic} (\Phi_{f,p}^{\acute{e}t}) + \Dim \text{Sw}(\Phi^n(\mathbb{F}_\ell)_p) = (-1)^n \mu_{f,p} .
		\end{equation}
		The second summand is the Swan conductor. The formula was proven in \cite{SGA7II} for the case of an isolated quadratic singularity, and the case of equal characteristics. Since then it has been proven in greater generality by Bloch \cite{Blo} (who suggested a more general conjecture involving the localized Chern class), Kato-Saito \cite{KS}, and in the full generality of the original conjecture by Beraldo-Pippi \cite{BP}.
		
		\item[(b)] \textbf{Monodromy}  \cite[Expos\'e XV -- \emph{La formule de Picard--Lefschetz}]{SGA7II}. We have an action of the inertia group $I= \Gal(\bar{\eta} / \eta)$ on the nearby cycles complex $\Psi_f \Z_\ell$. In the case of an ordinary double point the monodromy can be computed explicitly, the result depends on the parity of the dimension of $f$. The monodromy action defines also the monodromy operator
		     \begin{equation} \label{eq:Netale}
		         N \colon  \Psi^{\acute{e}t}_f \Z_\ell \to \Psi^{\acute{e}t}_f \Z_\ell(-1) ,
           \end{equation}
           where $\sigma \in I $ acts as $\text{exp}(Nt_\ell(\sigma))$, and $t_l: I \to \Z_\ell(1)$ is the $l$-th tame character. While Deligne's original proof for the Picard--Lefschetz formula included transcendental methods for the odd case, Illusie later gave a proof which is purely algebraic, \cite{Ill}.
	\end{itemize}
	
	\subsection{In motivic homotopy theory -- Ayoub}
	Joseph Ayoub constructed a refinement of Deligne's functor for the category of $\Aa^1$-motivic spectra $\SH(-)$,
	\[
	\Psi_f \colon \SH(X_\eta) \to \SH(X_\sigma).\]
	This then induces a functor also on $\DM(-)$, $\DM_\mathbb{Q}(-)$, etc. While the construction is more involved than Deligne's, it satisfies similar formal properties of a \emph{specialization system}. We list some of those properties and show how they help compute the nearby cycles. We use freely the six functors on 
	$\SH(-)$, for an account on this see \cite{Hoy}.
	\begin{Proper}
			\begin{itemize}
			\item  \label{property} Nearby cycles commute with smooth pullback and proper pushforward, that is, if we have a morphism $h$, in $\Sch_S$,
			\[\begin{tikzcd}[column sep=tiny, row sep=large]
				Y \arrow[rr, "h"] \arrow[rd, "g"'] & & X \arrow[ld, "f"] \\
				& S & 
			\end{tikzcd} , \]
			then	we have the induced natural isomorphisms:
			
			 $ h_\sigma^* \Psi_f \simeq \Psi_g h_\eta ^* $ if $h$ is smooth, and $ h_{\sigma. *} \Psi_g \simeq \Psi_f h_{\eta. *}$ if $h$ is proper.
			\item We have $\Psi_{\id} \one_{\eta} \simeq \one_{\sigma} $ (for the unipotent nearby cycles \cite[Proposition 3.4.9]{Ay07a}, for the tame nearby cycles \cite[Lemma 3.5.10]{Ay07a}).
		\end{itemize}
	\end{Proper}

	\begin{Example}[blowup] \label{ex:blowup}
		Let us see how by the above properties we can use blowup to help us compute nearby cycles. 
		Say that we have $f: X \to S$, and $h : Y \to X $ is a blowup with center contained in the special fiber $X_\sigma$. We have then that $h_\eta \colon Y_\eta \to X_\eta$ is an isomorphism, and by the previous property on proper pushforward,
		\[
		\Psi_f \one_{X_\eta} \simeq \Psi_f h_{\eta, *} h_\eta^* \one_{X_\eta} \simeq h_{\sigma. *} \Psi_g \one_{Y_\eta}.
		\]
		So if $\Psi_g \one$ is easier to compute than $\Psi_f \one$, we can use the former to compute the latter.
	\end{Example}
	
	\paragraph{Recent work} 
	In \cite{BGV}, the authors establish an equivalence between motives with monodromy, and rigid analytic motives. In \cite{Pr} the author proves a motivic version to Grothendieck's local monodromy theorem.
	In these notes we survey motivic enhancements for Expos\'e XV and Expos\'e XVI of \cite{SGA7II}, the Euler characteristic of nearby cycles and the monodromy operator. 
	For (a) we will use the \emph{quadratic Euler characteristic} and \emph{quadratic Milnor number}, in $\S~\ref{sec:3}$, \cite{Le20b}, \cite{Az}. For (b), We will use the \emph{motivic monodromy} and compute it, in $\S~\ref{sec:4}$, \cite{AJ}. For both results we will use a semistable reduction construction which works for homogeneous and quasi-homogeneous isolated singularities, see $\S~\ref{sec:2}$.

\section{Isolated singularities and semistable reduction} \label{sec:2}

We give our definitions for a homogeneous and quasi-homogeneous singularities that we use in this paper, and present constructions of semistable reduction to these singularities, showing how they help us compute nearby cycles.

\subsection{Semistable reduction}
	
	\begin{Def}
		\label{definition:sss}
		Let $f\colon X \to S$ be a morphism satisfying to Assumption~\ref{assumption}.
		\begin{enumerate}
			\item
			We say that $f$ is \emph{semistable} if the special fiber, $X_\sigma$, is a simple normal crossing divisor, in particular, reduced.
			\item	
			Let  $r$ be a natural number and define $r \colon S' := S[t^{1/r}] \to S $. We also use $r$ for the base change morphism
			$X_r := X\times_{S}S' \to X$,
			and $f_r \colon X_r \to S'$.
			We say that $f$
			\emph{admits semistable reduction}
			if
			there is a natural number $r$, 
			and a proper map
			$\pi \colon Y \to X_r $
			such that
			$g := f_r \circ \pi \colon Y \to S'$
			is a semistable morphism.
		\end{enumerate}
	\end{Def}
	
	When $f$ admits semistable reduction $g$, then in a similar manner to Example~\ref{ex:blowup}, we can reduce computing the nearby cycles spectrum on $f$, to the nearby cycles spectrum on $g$.
	
	\begin{Prop}[{\cite[Proposition~3.10]{Az}}] \label{prop:SStableRed}
		Assume $f:X \rightarrow S$ admits a semistable reduction $Y \xrightarrow{\pi} X_r \xrightarrow{f_r} S'$ for some $r$. {Let $r:X_r\to X$, and let  $g=f_r \circ \pi $.} Then \[ \Psi_f\one \simeq (r\circ \pi)_{\sigma*}\circ \Psi_{g} \one  . \]
	\end{Prop}
	
	\begin{proof}   	By \cite[Proposition 3.5.9]{Ay07a}, we have the natural isomorphism $\Psi_f \simeq r_{\sigma*}\circ\Psi_{f_r} \circ {r}_\eta^* $. Since $r_{\eta}$ is an isomorphism, the natural map $\id \to r_{\eta*}\circ r_{\eta}^*$ is an isomorphism. This, together with the pushforward property of $\Psi$ for proper maps, Property \ref{property}, gives the sequence of isomorphisms,
		\[
		\Psi_f(\one_{X_\eta}) \simeq r_{\sigma*}\circ\Psi_{f_r} (\one_{X_{e\eta}}) \simeq
		r_{\sigma*}\circ\Psi_{f_r} \circ \pi_{\eta*}\circ \pi_{\eta}^* \one_{X_{e\eta}}
		\simeq
		r_{\sigma*}\circ \pi_{\sigma*}\circ \Psi_{g}  (\one_{Y_\eta}) \simeq   (r \circ \pi)_{\sigma*}\circ \Psi_{g} (\one_{Y_\eta}). 
		\]
		\end{proof}

	A nice fact about semistable schemes, is that we can compute nearby cycles on strata of the special fiber by extending from the smooth locus.

	\begin{Prop}[{\cite[Th\'eor\`eme 3.3.44]{Ay07a}}, {\cite[Remark~3.4]{Az}}]\label{cor:computability}
	Let $f : X \rightarrow S$ be a morphism satisfying Assumption~\ref{assumption}, which is semistable, with $D_1, \ldots, D_r$ the irreducible components of $X_\sigma$. Fix an $i$ and let $D := D_i$, $A: = D \setminus \cup_{j\neq i}D_j$. We use the following notation, 
	\[A \xhookrightarrow{v} D \xhookrightarrow{u} X_\sigma
	\]
	for the respective open and closed immersions. We then have:
	\[
	u^* \Psi_f \one \simeq v_{*} \one_{A} = h_D(A)
	,
	\quad
	u^! \Psi_f \one \simeq v_{!} \one_{A} = h^c_D(A)
	.
	\]
	\end{Prop}
	
	\begin{proof}
		First, we use a theorem by Ayoub which tells us how to compute the nearby cycles on a stratum of the special fiber by restricting to the smooth locus. In our notation, composing $u^* \Psi_f f_\eta ^*$ with the unit map  $id \rightarrow v_* v^* $ of the adjunction, induces {a natural isomorphism} (\cite[Th\'eor\`eme 3.3.44]{Ay07a})
		\[
			u^* \Psi_f f_\eta^*
			\xrightarrow{\sim} v_{*} v^* u^* \Psi_f f_\eta^* 
			.
		\] 
Applying the above isomorphism on $\one_\eta$, and get
	\[
		u^* \Psi_f \one
	\simeq v_{*} v^* u^* \Psi_f \one
	\]
		Now, let  $w\colon A \to \sigma$ be the structure morphism. Use the compatibility with smooth pullback
	(Property ~\ref{property}),
	first to the open immersion
	$X \setminus \cup_{j\neq i} D_j\hookrightarrow X$,
	then to the smooth morphism
	$X \setminus \cup_{j\neq i} D_j \to S$,
	and using the second Property~\ref{property} on $\Psi_{\id} \one$,
	we get
	\[
	v^* u^* \Psi_f \one \simeq w^* \Psi_{\id} \one \simeq \one_A
	.
	\]
	Combining the two
	we obtain the first equivalence.
	The second equivalence is handled similarly, with the dual identity in \cite[Th\'eor\`eme 3.3.44]{Ay07a}.
	\end{proof}
	
	\subsection{Homogeneous singularities}
	
	\begin{Def}\label{def:homogeneoussingularity}
		Let $f\colon X \to S$ be a morphism satisfying Assumption~\ref{assumption}, with base field $k$ and parameter $t$ (so that either $S= \Spec k[t]$, or $S = \Spec \OO$ and $t$ is a fixed uniformizer in the dvr $\OO$).
		Let $F \in k[T_0,...,T_n]$ be a homogeneous polynomial of degree $r$,
		with $(r,\operatorname{char} k)=1$.
		We say that a singular point $p \in X_\sigma$ is
		\emph{a homogeneous singularity defined by $F$,}
		if
		the hypersurface $V(F) \subset \P_k^n$ is smooth,
		and we have
		\[ f^*(t) = F(x_0,...,x_n) \]
		in the local ring $\OO_{X,p}$ modulo $m^{r+1}$,
		where $(x_0,...,x_n)$ is a regular sequence
		generating the maximal ideal $m \subset \OO_{X,p}$ (local \'etale coordinates).
	\end{Def}
	
	Let $F \in k[To,...,T_n]$ be a homogeneous polynomial of degree $r$. We define the following embedding of schemes,
	\begin{align} \label{clopimmersion}
	C:= \{T_{n+1} = 0\} \xhookrightarrow{i} D:=V_{\P^{n+1}}(F-T_{n+1}^r) \xhookleftarrow{j} A : = \{T_{n+1} \neq 0\}.
	\end{align}
	
	We claim that all the data about the nearby cycles spectrum at a homogeneous singularity $p$, defined by a polynomial $F$ (including the monodromy operator), can be deduced from this embedding of schemes. Notice that by the localization property we have the following natural maps:
	\begin{align} \label{eq:alpha}
		 \alpha : j_* \one_A \to i_* i^! \one_D [1] \\
		 \beta : i_* \one_C \to j_! \one_A [1] . \label{eq:beta}
	\end{align}
	Those two will appear later when we describe the monodromy.
	 To begin, we make the following claim which computes the nearby cycles spectrum restricted to the singular point.
	
	\begin{Prop}[{\cite[Theorem~4.3]{Az}}, {\cite[Proposition~3.2]{AJ}}] \label{prop:nearbypoint}
				Let $f\colon X \to S$ be 
		a flat, quasi-projective morphism, with a homogeneous singularity $o \colon p \hookrightarrow X_\sigma$ defined by a polynomial $F(T_0,...,T_n)$ of degree $r$, and let $A: = V(F-T_{n+1}^r) \setminus V(T_{n+1}) \subset \P^{n+1} $.
		We have natural isomorphisms
		\begin{align}
		o^*\Psi_f\one \simeq h(A),  \\   o^! \Psi_f \one\simeq h_c(A) .
	\end{align}
	\end{Prop}
	
	\begin{sketch}
		First we make the following construction:
			\begin{enumerate}
			\item Blow up $f \colon X \to S$ at the singular point \(p\), obtaining $\Bl_p(X) \to S$ .
			\item Pass to the base change \( S': = S[t^{1/r}] \to S \), $\Bl_p(X)_r \to S'$.
			\item Normalize the result  to obtain a morphism $g \colon Y \to S'$.
				\end{enumerate}		
			Computing the blowup by the standard covering shows that $g$ is a semistable reduction for $f$. Moreover, the special fiber $Y_\sigma$, is the union of two smooth divisors $Y_\sigma := D  \cup D'$, where $D$ is the preimage of the exceptional divisor in the blowup, and $D'$ the preimage of the strict transform. In addition, define $C := D \cap D'$. Then the close immersion $u \colon C \hookrightarrow D$ is isomorphic to the close immersion $i$ in Equation~\eqref{clopimmersion} \cite[Theorem~4.3]{Az}. 
			Now by the Proposition~\ref{prop:SStableRed}, we have 
			\[ o^* \Psi_f\one \simeq o^* (r\circ \pi)_{\sigma*}\circ \Psi_{g} \one .\]
			Let $w$ be the restriction of $(r\circ \pi)_{\sigma}$ to $D$. Then by proper base change, $o^* (r\circ \pi)_{\sigma*} =w_* u^*$, and we get $o^* \Psi_f\one \simeq w_* u^*\Psi_{g} \one$. Finally,
			since $g$ is semistable, and setting $A: = D \setminus C \simeq  V(F-T_{n+1}^r) \setminus V(T_{n+1})$, we have by Proposition~\ref{cor:computability},
			\[u^* \Psi_g \one \simeq h_D(A) ,\]
			and so we get the first statement in the proposition, the second one is obtained in a similar way.
	\end{sketch}
		
	\begin{Prop}[{\cite[Proposition~3.7]{AJ}} for $\DM(-)$] \label{prop: triangle}
		Let $g: Y \to S$ be a semistable morphism, with $Y_\sigma = D \cup D'$, $D$, $D'$ smooth, $C: = D \cap D'$, and $u \colon D \hookrightarrow Y_\sigma$ the closed immersion.	We have the following fiber sequence in $\SH(D)$:
	\begin{align} \label{eq:triangle_a}
	\one_{D}\to  u^* \Psi_g \one\xrightarrow{} i_* \Sigma^{-\mathcal{N}_{i}} \one_{C} [1] ,
	\end{align}
	where $\mathcal{N}_i$ is the normal bundle of the closed immersion $i \colon C \hookrightarrow D$.
	\end{Prop}
	
	\begin{proof}
	Consider the localization sequence in $\SH(D)$ with $C \xhookrightarrow{i} D \xhookleftarrow{j} A$:
	\[
	 i_! i^! \to \id \to j_*j^* 
	. \]
	Applying this to $\one_D$, and using purity along closed immersions (\cite[Proposition 5.7]{Hoy}),  which gives $i^! \one_{D} \simeq \Sigma^{-\mathcal{N}_{i}} \one_{C}$, we get the sequence
	\[
	i_* \Sigma^{- \mathcal{N}_{i}} \one_C \to \one_{D} \to j_* \one_A .
	\]
	Now using the result of Prop~\ref{prop:nearbypoint} and shifting, we get the first fiber sequence.
	\end{proof}
	
	\begin{Rem}
		While the twist $\Sigma^{-\mathcal{N}_{i}}$ in general depends on the embedding $i$, its realization in cohomology theories that come with some type of orientation data, depends on less info. For example, if we look at $H \Z$, Chow group cohomology (or equivalently if we work in $\DM(-)$), this depends only on the codimension of $C$ in $D$, which is $1$, and then $\Sigma^{-\mathcal{N}_{i}}(-) = (-)\otimes (-1)[-2]$, 
	\end{Rem}
	
		We have a natural map $K_0(\Var_k) \to K_0 (\SH(k))$, defined by $[p \colon X \to k] \mapsto [p_! \one_X] $, so we identify $[X] = [p_! \one_X] \in K_0(\SH(k))$.
	
	\begin{Cor}[{\cite[Corollary~4.4]{Az}}] \label{cor:K0id}
			Let $f\colon X \to S$ be 
		a morphism satisfying Assumption~\ref{assumption}, with a homogeneous singularity $o \colon p \hookrightarrow X_\sigma$ defined by a polynomial $F$, Def.~\ref{def:homogeneoussingularity}. We have the following identity in $K_0(\SH(k(p))$, 
		\[ [\Psi_f|_p] = [D] - [\Aa^1]\cdot[C],
		\]
		with the notation of \eqref{clopimmersion}.
	\end{Cor}
	
	\begin{proof}
		First, by the same construction as in the proof of Proposition~\ref{prop:nearbypoint}, we get $g \colon Y \to S'$, with $Y_\sigma = D \cup D'$, and the close immersion $u \colon C:= D \cap D' \hookrightarrow D$ isomorphic to the close immersion $i$ in Equation~\eqref{clopimmersion}, and $D$ the preimage of $p$ under the construction. Letting $w \colon D \to \Spec k(p)$, we have $\Psi_f \one |_p = o^* \Psi_f \one \simeq w_* u^* \Psi_g \one $.
		Now, we pushing forward Equation~\eqref{eq:triangle_a} by $w$, and using the fact that $\mathcal{N}_i$ is Zariski locally trivial, i.e., it is locally isomorphic to $\Aa^1_C$, and we have Mayer-Vietoris property in $K_0(-)$ (see \cite[Proposition~2.12]{Az} for more details), we get the result.
	\end{proof}
	\begin{Rem}
		Corollary~\ref{cor:K0id} follows immediately from a more general result by Ayoub-Ivorra-Sebag, \cite[Theorem~8.6]{AIS}. It states that the class of $\Psi_f \one$ in $K_0(\SH_c(X_\sigma))$, is equal to the motivic Milnor fiber in the sense of Denef-Loeser, which is considered in the theory of motivic integration \cite{DL}. A proof of the formula in the special case of no triple intersections, with a construction similar to the one in Proposition~\ref{prop:nearbypoint} appears in \cite{Az}.  
	\end{Rem}
	
	\subsection{Quasi-homogeneous singularities}
		We also deal with a more general class of singularities, where we can put different weights on the variables. Here there are some additional assumptions detailed below.
	
	\begin{Def}\label{def:qhomogeneoussingularity}
		Let $\underline{a}=(a_1,...,a_n)$ a vector of natural numbers and let $\P(\underline{a})$ be the weighted homogeneous space $\operatorname{Proj} k[T_0,...,T_n]$, with the ring graded by letting $T_i$ have degree $a_i$. Let $F \in k[T_0,...,T_n]$ be an $\underline{a}$-weighted homogeneous polynomial of degree $r$. It defines a hypersurface $V(F)$ in the weighted homogeneous space $\P(\underline{a})$. We require that the weights $a_i$ are pairwise relatively prime, each $a_i$ divides $r$, and $r$ is prime to the exponential characteristic of $k$. Let $G(y_0,...,y_n): =F(y_0^{a_0},..,y_n^{a_n})$ and assume that both $V(G) \subset \P_k^n$, and $V(F) \subset \P(\underline{a})$ are smooth. Furthermore, letting $v_i\in   \P(\underline{a})$ be the point with the $i$-th homogeneous coordinate 1, and all other coordinates 0,  we require that $F(v_i)\neq0$ if $a_i>1$ (this last condition is superfluous in the case $n > 1$, see \cite[Remark 5.4]{Az}). \\
		Let $f\colon X \to \Spec k[t]$ be a flat, quasi-projective morphism.	We say that a  singular point $p \in X_\sigma$ is \emph{an $\underline{a}$-weighted homogeneous singularity defined by $F$}, if in the local ring $\OO_{X,p}$, we can write
		\[ f^*(t) = F(x_0,...,x_n) + h,\]
		where $\underline{x}=(x_0,...,x_n)$ is a regular sequence generating the maximal ideal $m \subset \OO_{X,p}$, and $h \in m \cdot m_{\underline{a}}^{(r)}$, where $m_{\underline{a}}^{(r)}$ is the ideal generated by monomials in $\underline{x}$ of $\underline{a}$-weighted degree $r$. For more details see \cite[Definition 5.1, Definition 5.3]{Az}.
	\end{Def}
	
	\begin{Prop}\label{proposition:qhomogeneousssred}
		Let $f\colon X \to S$ be
		a flat, quasi-projective morphism.
		Suppose that $X_\sigma$ has a single quasi-homogeneous singularity $o$
		defined by $F \in k[T_0, \dotsc, T_n]$ of degree $r$.
		Then $f$ is admits semistable reduction
		$g \colon Y \to \Aa^l$,
		where $g = f_r \circ \pi$,
		$f_r \colon X_r \to \Aa^1$
		is $r$-th power base change of $f$,
		and $\pi \colon Y \to X_r$ is proper.
		We have $Y_\sigma = D \cup D'$, with
		$D \simeq V_{\P(\underline{a},1)}(F-T_{n+1}^r)$,
		and
		$C : = D \cap D' \simeq V_{\P(\underline{a})}(F)$.
	\end{Prop}	
	
	\begin{proof}
		This follows from \cite[Thm.~5.7]{Az}.
		While the construction of $g$ is more involved than in
		Proposition~\ref{prop:nearbypoint}
		(replacing the blowup by a certain construction of a weighted blowup),
		it still satisfies the statement of the proposition.
	\end{proof}
	
	\begin{Rem}
		With similar complementing closed and open immersions to \eqref{clopimmersion}, only in weighted projective space, more precisely:
		\[
			C:= \{T_{n+1} = 0\} \xhookrightarrow{i} D:=V_{\P(\underline{a}, 1)}(F-T_{n+1}^r) \xhookleftarrow{j} A : = \{T_{n+1} \neq 0\} ,
			\]
			we get quasi-homogeneous generalizations for most of the results in these notes: Theorem~\ref{th:Az}, Theorem~\ref{th:LPLS}, Theorem~\ref{th:chiformula} and Theorem~\ref{th:abstractthm}.
	\end{Rem}
	
	\section{The quadratic Euler characteristic of nearby cycles} \label{sec:3}

	Here we have a motivic version for the Euler characteristic of the Milnor fiber, and Milnor's Formula. The results in this section are due to the work of Levine, Srinivas and Pepin Lehalleur \cite{Le20b}, and the work of the author \cite{Az}.
	
	\subsection{Quadratic invariants}
	
	We present here quadratic refinements to the invariants appearing in the Milnor and Deligne--Milnor formula \eqref{eq:Milnor}, \eqref{eq:DeligneFormula}. Quadratic refinements take value in the \emph{Grothendieck--Witt ring of the base field}, $\GW(k)$, and refine classical invariants in the ring of integers $\mathbb{Z}$. We work over a base field $k$ with characteristic different than $2$, and denote by $ \<a\> \in \GW(k) $ the element corresponding to the quadratic form $x \mapsto ax^2$. 
	\paragraph{The quadratic Euler characteristic}
	Let $p \colon S \to k$ be a morphism, and  $\alpha \in \SH(S)$ be a dualizable motivic spectrum in the symmetric monoidal category $(\SH(S), \otimes, \one_k)$. Assume here $\Char k = 0$ for simplicity, otherwise we may have to invert the characteristic. The \emph{compactly supported quadratic Euler characteristic}, is defined, invoking the notion of trace in a symmetric monoidal category (\cite[Definition 2.2]{PS}), as
	\[ \chic(\alpha / p) : = \tr(\id_{p_!\alpha}) \in \End_{\SH(k)}(\one_k) \simeq \GW(k) . \]
	 For a scheme $f \colon X \to k$ we define $\chic(X/k) : = \chic(\one_X / f)$. This factors as a \emph{motivic measure},
	\[ \chic \colon K_0(\Var_k) \to \GW(k) ,\]
	for details see also \cite[\S~2]{Le20a}, \cite[\S~2]{Az}, \cite{BM}. If $k$ embeds in $\mathbb{C}$, taking rank recovers the topological Euler characteristic of $X(\C)$, similarly we can take \'etale realization and recover the $\ell$-adic Euler characteristic. Extending to all dualizable motivic spectra, by the additivity of traces along fiber sequences (\cite{May}), we get a ring homomorphism
	\[ \chic \colon K_0(\SH^c(k)) \to \GW(k). \]
	In particular, $\chic$ is well defined on our version of motivic Milnor fiber, i.e. the motivic nearby cycles spectrum, $\Psi_f \one$, and also for the motivic vanishing cycles spectrum $\Phi_f \one$, as well as their restriction to any closed subscheme of the special fiber.
	
	Let $f \colon X \to S = \Spec \OO$ where $\OO$ is a dvr with uniformizer $t$, $K$ is the function field and $k$ is the residue field. We have a ring homomorphism
	\[ \Sp_t \colon \GW(K) \to \GW(k) \]
	sending $\<t^nu\>$ to $\<\bar{u}\>$ ,where $u \in \OO^\times$.
	If $f$ is proper, we have \cite[Prop. 8.1]{Le20b},
	\begin{equation} \chic (\Psi_f \one) = \Sp_t \chic(X_\eta)  . \end{equation}
	This is a motivic version for the formula \eqref{eq:properetale}, and uses the smooth pullback in Property~\ref{property}. We can now define the following differences.
	 \begin{Def}
		Let $f \colon X \to S = \Spec \OO$, according to Assumption~\ref{assumption}.
		\begin{itemize}
			\item Let $p \in X_\sigma$ be a closed point. Define
			\[
			\Delta \chi(f)_p : =  \chic(\Psi_f \one |_p) - \<1\> .
			\]
			\item 
			For $f$ proper, define
			\[
			\Delta\chi(f) : = 	\Sp_t\chic(X_\eta)-\chic(X_\sigma) . \]
		\end{itemize} 
		
		\begin{Rem} \label{rem:deltasmooth}
			Note that if $f$ is smooth, $\chic (\Psi_f \one) = \chic (X_\sigma)$, and $\Delta \chi (f) = 0$. 
			\end{Rem}
		
	\end{Def}
	
	\paragraph{The quadratic Milnor number} We define a quadratic version to the Milnor number at a point. Given $f: X \to S$ and $p \in X_\sigma$ an isolated singularity, we may consider the section $df \colon X \to \Omega_{X/k}$, so that $p$ is an isolated zero of $df$. Then we can define the \emph{quadratic Milnor number} as the Euler class with respect to this section, taken in Hermitian K-theory, or equivalently in Chow-Witt theory. Below is the definition of the Euler class with respect to a section of a vector bundle.
	 \begin{Def} \label{LocalEulerClass}
		Fix a motivic ring spectrum $E \in \SH(S)$. Let $V \rightarrow X$ be a vector bundle of rank n, let $s: X \rightarrow V$ be a section, and $i: Z  =Z(s) \hookrightarrow X $ the zero locus of $s$. The \emph{local Euler class} of $(V,s)$, is the element $e(V,s) \in E_Z^{V^*}(X)$  defined by the composition
		\[ X/X \setminus Z \xrightarrow{s} V/V\setminus 0 \simeq \Sigma^{V^*} \one_k \rightarrow \Sigma^{V^*} E|_X \in \SH(X) . \]  
	\end{Def}
	\begin{Def}
	Let $f \colon X \to S$ be a morphism, with $p$ an isolated critical point. The \emph{quadratic Milnor number} $\mu_{f,p}$ is the local Euler class with respect to the vector bundle $V =  \Omega_{X/k}$ and the section given by $df$, (where the zero locus is the isolated point $p$). $E$ can be taken to be the spectrum representing Hermitian $K$-theory or Milnor-Witt sheaves. 
		\[
	\mu_{f,p}^q : = e^{\GW}_p(\Omega_{X/k}, df) \in \GW(k(p)).
	\] 
	\end{Def}
 This definition is a motivic version to Bloch's localized Chern class (\cite{Blo}), and could be used also if the critical locus were not isolated.
	For a more  explicit definition, $\mu_{f,p}^q$ is given by a canonical bilinear form $[B_{f,p}]$ on the Jacobian ring at $p$. 
	We have the \textit{Jacobian ring of $f$ at $p$}, $J(f,p)$, 
	\[
	J(f,p):=\OO_{X,p}/(\partial f/\partial s_1\ldots \partial f / \partial s_n) ,
	\] 
	where $(s_1,...,s_n)$ are local coordinates of $X$ at $p$. Now the form
	\[B_{f,p} : J(f,p) \times_{k(p)} J(f,p) \to k(p).\]
	 satisfies $ B_{f,p}(x,y) = 1$ if $xy = e_{f,p}$, where  $e_{f,p} \in J(f,p)$ is a special element in the Jacobian ring, called the {\em Scheja-Storch element}. This is the class of \emph{Eisenbud-Khimshiashvili-Levine} \cite{KW}. For more details, see \cite[Proposition~2.32, Theorem~7.6]{BW}, \cite[Corollary~3.3]{Le20a}, and \cite[\S~6.3]{Az}.

	If $k(p)$ embeds in $\C$, then by taking the rank of the corresponding quadratic form we recover the classical Milnor number, $\rk \mu^q_{f,p} = \dim J(f,p) = \mu_{f,p}$.
	
	\subsection{The formulas} 
	
	\begin{Th}[{\cite[Corollary~4.4]{Az}}]  \label{th:Az}
		Let $f: X  \to \Spec \OO $ be a morphism with Assumption~\ref{assumption}, and $p \in X_\sigma$ an homogeneous singularity of the special fibre $X_\sigma$ defined by a  $F\in k(p)[T_0,\ldots, T_n]$ of degree $r$, Def.~\ref{def:homogeneoussingularity}. Then, with the notation of \eqref{clopimmersion},
		\[
		\chic(\Psi_f\one|_p)= \chic(D)-
		\<-1\> \chic(A).
		\]
	\end{Th}
	
	\begin{proof}
		This follows from Corollary~\ref{cor:K0id}, as $\chic (\Aa^1) = \< -1 \>$ .
			\end{proof}
	
	\begin{Th}[Hypersurface quadratic conductor formula, {\cite[Theorem~5.5]{Le20b}}] \label{th:LPLS}
		Use notation as in Assumption~\ref{assumption}, with $\OO$ a dvr, and $t \in \OO$ a uniformizer.
		Let $F \in \OO[T_0, ..., T_n]$ be a homogeneous
		polynomial of degree $r$ and let $\mathcal{X}_F \subset \P^{n+1}_\OO$
		 be the hypersurface defined by the homogeneous
		equation
		$F(T_0,...,T_n) − tT_{n+1}^r$. Let $f_F$ be the morphism $ \mathcal{X}_F  \to \Spec k[t]$.
		Suppose that $V(\bar{F}) \in \P^n_k$ is smooth, and let $o$ be the vertex of the projective cone $ f^{-1} (0)$. Then
			\[\Delta\chi(f_F ) =   \< r\> - \<1\> + (-\<r\>)^{n} \cdot \mu_{f_F,o}^q \]
			in $\GW(k)$.
		\end{Th}
	\begin{sketch} We will not give many details as this is the content of the majority of the paper \cite{Le20b}.
		The formula computes the difference of Euler characteristics of two projective hypersurfaces. For the Euler characteristic of a projective hypersurface it is possible to use the motivic Gauss-Bonnet formula, which describes $\chic(-)$ in terms of a symmetric bilinear forms on Hodge cohomology groups, \cite{LR}. In the case of a projective hypersurface, the graded ring of Hodge cohomology can be compared to the Jacobian ring, this is done in \cite[\S~3]{Le20b}. This way we get an expression for $\chic(-)$ of a projective hypersurface in terms of the quadratic Milnor number. To get the exact values for the conductor formula one may take the rank, and compare with the classical Milnor formula. 
	\end{sketch}
	
	\begin{Th}[Generalized quadratic conductor formula {\cite[Corllary~7.3]{Az}}] \label{th:chiformula}
			Let $f: X \rightarrow \Spec \OO$ be according to Assumption~\ref{assumption}, of relative dimension $n$, {with $f$ proper} and $k$ of characteristic $0$. Suppose that $X_\sigma$ is smooth away from isolated singular points $\{p_1,...,p_s\}$, where each $p_i$ is a homogeneous singularity, defined by the polynomial $F_i$ of degree $r_i$, Definition~\ref{def:homogeneoussingularity}. Then
			\[
			\Delta\chi(f) = \sum_i \Tr_{k(p_i)/k} \left( \< r_i\> - \<1\> + (-\<r_i\>)^{n} \cdot \mu_{f,p_i}^q \right) .\]
	\end{Th} 
	\begin{sketch}
		By the smooth pullback property of nearby cycles (see Remark~\ref{rem:deltasmooth}), we have 
		\[
		\Delta\chi(f) = \sum_i \Delta \chi (f)_{p_i}.
		\]
		 Now for each $i$, $\Delta \chi (f)_{p_i}$ can be computed using Theorem~\ref{th:Az}, and interpreting $D$ and $C$ according to \eqref{clopimmersion}, we get 
		\[
		\Delta \chi (f)_{p_i} = \Tr_{k(p_i)/k} (\Delta \chi (f_{F_i})).
		\]
		 We then prove that the quadratic Milnor numbers $\mu^q_{f,p}$ and $\mu^q_{f_F,o}$ are equal by a homotopy invariance argument in \cite[\S~6]{Az}. Then we conclude using Theorem~\ref{th:LPLS}.
	\end{sketch}
	\begin{Rem}
		Comparing with the classical formulas notice the following differences: First, since we work only with singularities whose order is prime to the characteristic, the Swan conductor in the classical formula \eqref{eq:DeligneFormula} is $0$, so we do not have to deal with in the case we are considering. Second, in our formula appear some correction terms which depend on the degree $r$ of the singular point, and have no remnant in the classical formulas: the term $ \< r\> - \<1\> $ vanishes when we take complex or real realizations (via the rank and signature respectively), and the term $(-\<r\>)^n$ equals $(-1)^n$ under those realizations, as in the classical formulas. It remains to see if simpler expressions (as we get after cancelling out terms as in Example~\ref{example}) are valid for some class of maps $f$, possibly also when the singular locus is not isolated.
	\end{Rem}
	\begin{Ex}  \label{example}
			\begin{enumerate}
		\item Let $p$ be an ordinary double point, defined by a polynomial $F = x^2 - y^2$.  We have then $J(f,p) \simeq k$, $\mu^q_{f,p} = \<-1\>$, $r=2$, $n=1$, and so we get
		\[
		\Delta \chi (f)_p = \<2\> - \<1\> -  \<2\>\<-1\> = - \<-1\> .
		\]
		Assuming $p$ is the only singular point, this is also the value of $\Delta \chi (f)$.
		\paragraph{Exercise}(suggested by Yonatan Harpaz). Compute this directly as a difference of Euler characteristics, without using the formulas above. However, you may use the complex realization via rank, and the formula for topological Euler characteristic of a curve by the genus.
		\item Let $p$ be defined by $F=x^2 -y^3$. This is a quasi-homogeneous singularity, putting as weights $(a_1,a_2) = (3,2)$.
		The formula then is  $ \Delta \chi (f)_p = \< a_1 a_2 r\> - \<1\> + (-\<  a_1 a_2 r\>)^{n} \cdot \mu_{f,p}^q $. In this case we have $r =6$, $n =1$, and $\mu_{f,p} = \<1\>+ \<-1\>$ (see \cite[Example~5.3]{Orm}). So we have
		\[
		\Delta \chi (f)_p = \<36 \> -\<1\> -\<36\> \<-1\> = - \<-1\>.
		\]
	\end{enumerate}
	\end{Ex}
	
	\section{The motivic monodromy} \label{sec:4}
	
	This section is due to the work of Emil Jacobsen and the author \cite{AJ}.
	\paragraph{The monodromy and the variation}In this section we work in category $\DM_\mathbb{Q}(-)$ of $H \mathbb{Q}$ modules in $\SH(-)$; $\Psi_f$ here denotes the unipotent version of motivic nearby cycles (denoted by $\Upsilon_f$ in Ayoub's papers). Our $f \colon X \to \Aa^1_k$ satisfies Assumption~\ref{assumption}, and $k$ is of characteristic $0$.
	We have the following fiber sequence (\cite[Theor\'eme~3.6.46]{Ay07a})
	\[ \bar{i}^* \bar{j}_*\to \Psi_f \xrightarrow{N} \Psi_f(-1) .\]
	
	The map $N_f : \Psi_f \one \to \Psi_f \one (-1)$ is the \emph{motivic monodromy operator}, a motivic enhancement for \eqref{eq:Netale}. Ayoub proved that just like the classical case, $N_f$ is nilpotent \cite[Corollaire~3.6.49]{Ay07a}.
	
	Let $c: C \hookrightarrow X_\sigma$ be the singular locus of $f$. The monodromy is trivial away from it, therefore it factors through $C$ in the following manner:
	\[
	N_f: \Psi_f \one \xrightarrow{\eta_C} c_*c^* \Psi_f \one \xrightarrow{c_* \var_f} c_* c^! \Psi_f \one (-1) \xrightarrow{\epsilon_c} \Psi_f \one (-1) . 
	\] 
	Since $c_*$ is fully faithful, we have the well defined \emph{variation map}
	\[ \var_f : c^* \Psi_f \one \to c^! \Psi_f \one (-1) .\]
	As this captures the non-trivial part of the monodromy, all we have to do is to compute $\var_f$, which is our goal in this section. 
	
	\begin{Th}[The abstract Picard--Lefschetz Theorem, {\cite[Theorem~5.7]{AJ}}] \label{th:abstractthm}
		Let $f \colon X \to \Aa^1$ be a morphism, Assumption~\ref{assumption}, with a homogeneous singularity defined by a polynomial $F$ of degree $r$ (Def.~\ref{def:homogeneoussingularity}), then,
		\[\begin{tikzcd}
			o^*\Psi_f\one   \arrow[d, "\sim" ' {rotate=90, anchor=north}] 	\arrow[rr, "-r \cdot var"] & & o^!\Psi_f\one(-1) \\
			h( {A} )	\arrow[r, "\alpha"]  	&	h(C) (-1)[-1]  \arrow[r, "\beta(-1)"]
			&	h^c(A)(-1)  \arrow[u, "\sim" {rotate=90, anchor=south} ]  .
		\end{tikzcd}\]
		The vertical isomorphisms are those of Proposition~\ref{prop:nearbypoint}; $\alpha$ and $\beta$ are the natural maps \eqref{eq:alpha}, \eqref{eq:beta}.
	\end{Th}
	\begin{sketch}
		To make it more interesting, we describe the main steps of the proof in the reverse order of the way it is done in \cite{AJ}.
	\begin{enumerate}

			\item[Step I] First we compute the monodromy on the nearby Kummer motive (because we can).
			The unipotent nearby cycles functor can be constructed as a logarithm \cite[\S~3.6]{Ay07a}, \cite[\S~11]{Ay14}. This involves a construction over a certain $\G_m$ motive, the Kummer motive $\mathcal{K}$. It sits in a canonical fiber sequence,
			\[ \one_{\G_m} \to \mathcal{K} \to \one_{\G_m}(-1) . \]
			Then $\Psi_f(-)$ is given by $\bar{i}^*\bar{j}_* (f_\eta^* \Sym^\infty \mathcal{K} \otimes -)$.
			
			 Using this we can compute
			 \[ \Psi_{\id} \mathcal{K} \simeq \one_{\G_m} \oplus \one_{\G_m}(-1) , \]
			 and
			  \[N_{\id} \mathcal{K} : \Psi_{\id} \mathcal{K} \simeq \one \oplus \one(-1) \to  \one (-1) \oplus \one(-2) \simeq \Psi_{\id} \mathcal{K}(-1) \] to be given by  $\begin{pmatrix} 0 &-1 \\ 0 &0  \end{pmatrix}$, \cite[Proposition~5.6]{AJ}.
			\item[Step II] The monodromy on a two-branched semistable family: Using the computation above we can deduce the monodromy $N_g \one $ for the family \[
			g':  \Spec k[t,x,y]/(xy-t) \to  \Spec k[t] . \]
			$g'$ has generic fiber isomorphic to $\G_m \times \G_m \to \G_m$, while its special fiber is $\Spec k[x,y]/(xy)$. Computing $N_{g'}$ can be done as the pushforward $g'_{\sigma *} \Psi_{g'} \one$ is closely related to $\Psi_{\id} \mathcal{K}$ from the first step, see \cite[Corollary~4.9]{AJ}.
			
			But this $g'$ with its $1$-dimensional special fiber, can serve as model to all semistable maps $g \colon Y \to \Aa^1$ with two smooth branches in the special fiber, $Y_\sigma = D \cup D'$: by using the smooth pullback property, \eqref{property}, we get the monodromy for any such family, \cite[Proposition~3.5]{AJ}. 
			\item[Step III]  Let $g \colon Y \to \Aa^1$ be a semistable morphism, with the special fiber a union of two smooth divisors,  $Y_\sigma = D \cup D'$, and let $C:= D \cap D'$, and $A := D \setminus C$. we then get the following fiber sequences for $g$ (in a similar manner to Proposition~\ref{prop: triangle}, interpreted in $\DM_{\mathbb{Q}}(D)$, {\cite[Proposition~3.7]{AJ}}):
			\[ \one \to c_*c^* \Psi_g \one \simeq h_D(A) \xrightarrow{\alpha} h(C)(-1)[-1] ,\]
			and
			\[  h(C)(-1)[-1] \xrightarrow{\beta(-1)}  h_D^c(A) (-1) \simeq c_*c^! \Psi_g \one (-1)\to \one(-1) . \]
			Note that $c_* \var_g$ is connecting between the middle terms in both sequences. Using the fact that $Hom(\one, \one (-1)) = 0$ in $\DM(-)$, we obtain that $\var_f$ must factor through a map from the last term in the fist fiber sequence, to the first term in the second fiber sequence (which are isomorphic to each other). We have then the decomposition 			
					\[\begin{tikzcd}
				c_*c^*\Psi_g\one   \arrow[d, "\sim" ' {rotate=90, anchor=north}] 	\arrow[rrr, "c_*\var_g"] & & & c_*c^!\Psi_g\one(-1) \\
				h( {A} )	\arrow[r, "\alpha"]  	&	h(C) (-1)[-1] \arrow[r, "\lambda"] & h(C) (-1)[-1] \arrow[r, "\beta(-1)"]
				&	h^c(A)(-1)  \arrow[u, "\sim" {rotate=90, anchor=south} ]  .
			\end{tikzcd}\]
			$\lambda$ is a-priori an unknown endomorphism of $h(C)(-1)[-1]$. But in Step II we computed $\var$ for exactly the same $g$, so we can compare the two and get $\lambda = -1$.
			\item[Step IV]
			In order to compute the monodromy on our original $f: X \to \Aa^1$ 	we reduce to a two branched semistable $g$.
					The construction gives us the natural vertical identifications in the theorem, and using the result of the steps above we obtain $N_f$ and the required $var$, in terms of the maps $\alpha$ and  $\beta$, {\cite[Proposition~3.9]{AJ}}. 
		\end{enumerate}
		\end{sketch}
		\begin{Rem}The proof above follows Illusie's spirit in \cite{Ill}, both in the tools, turning to semistable reduction, and in the principle of a self contained proof without turning to realizations.
		\end{Rem}
		 Note that in the case $r=2$, $C$ and $D$ in the statement of the theorem are smooth projective quadrics, and $A$ is an affine quadrics. So computing their motives and the maps $\alpha$, $\beta$ between them, would give us an explicit description for the monodromy, enhancing Deligne's formula in \cite[\'expos\'e XV]{SGA7II}.
		
			\begin{Prop}[{\cite[Proposition~2]{Ro}}, {\cite[Proposition~6.4]{AJ}}] \label{prop:motivesquadrics}
				Let $k$ be an algebraically closed field of characteristic different than $2$. We then have the following results for projective and affine quadrics.
				\begin{enumerate}
				\item 
					Let $Q \subset \P^{n+1}$ be a smooth quadric of dimension n. Then
				\begin{equation*}
					h(Q) =
					\begin{cases}
						\bigoplus_{i=0}^n \one (-i)[-2i] &		\text{n odd} \\
						\bigoplus_{i=0}^n \one  (-i)[-2i]  \oplus \one  (-n/2)[-n] & \text{n even}
					\end{cases} .
					\end{equation*}
					\item
				Let $A$ be a smooth affine quadric of dimension $n$, i.e., $A$ is isomorphic to $Q \setminus H \subset \P^{n+1} $ where $Q$ is a smooth projective quadric of dimension $n$, $H$ is a projective hyperplane, and $ C:= Q \cap H$ is smooth. Then we have the following fiber sequences:
				\begin{align} \label{eq:motiveA}
			\one  \to	 h(A)  \xrightarrow{m_1} \one(-\lceil n/2 \rceil )[-n] \\
						\one(-\lfloor n/2 \rfloor)[-n] \xrightarrow{m_2}	 h_c(A)  \to	 \one(-n)[-2n].  \label{eq:motivecA}
				\end{align}
			\end{enumerate}
		\end{Prop}
		
		\begin{sketch}
			1. is a result by Rost in Chow motives \cite[Proposition~2]{Ro} that is still valid in $\DM(-)$.
			
			2. Follows from Rost's result and the localization sequences of the closed and open embedding $A \hookrightarrow Q \hookleftarrow C$, for details see \cite[Popostion~6.4]{AJ}.
		\end{sketch}
	
	\begin{Th}[The motivic Picard--Lefschetz formula, {\cite[Theorem~6.5]{AJ}}]
		Let $k$ be an algebraically closed field of characteristic different than $2$,
		and let $f \colon X \to \Aa^l_k$
		be a flat, quasi-projective morphism, of relative dimension $n$,
		with an isolated non-degenerate quadratic singularity $o$.
		
		Then, if $n$ is even, $\var_f$
		(and therefore the monodromy operator $N_f$) is the zero map.
		
		If $n=2m+1$ is odd,
		$\var_f$ factors as
		\begin{equation*}
			o^*\Psi_f \one
			\xrightarrow{m_1} \one(-m-1)[-n]
			\xrightarrow{-1} \one(-m-1)[-n]
			\xrightarrow{m_2} o^!\Psi_f \one(-1)
			.
		\end{equation*}
	\end{Th}
	\begin{sketch}
		The theorem follows from Theorem~\ref{th:abstractthm} and the explicit computations of Proposition~\ref{prop:motivesquadrics}
		We will give the proof for the even case. The odd case is similar but also requires the computation of the maps $\alpha$, $\beta$, see \cite[Proposition~6.4]{AJ}.
		In the even case, $n=2m$, the $\var$ map, when identifying the term as in Proposition~\ref{prop:nearbypoint}, is given by $\var_f \colon h(A) \to h_c(A)(-1) $. Precomposing with the first arrow of \eqref{eq:motiveA}, and composing with the last arrow of \eqref{eq:motivecA}, we get a map $ \one \to \one(-n-1)[-2n]$ which must be zero, as in $\DM(-)$ there are no maps to a negative twist. It follows that $\var_f$ factors through a map from the cofiber in  \eqref{eq:motiveA} to the fiber in  \eqref{eq:motivecA} twisted by $(-1)$, i.e., a map
		\[
		\one(-m )[-n] \to \one(m-1)[-n] .
		\]
		But this is zero for the same reason, so $\var_f = 0$.
	\end{sketch}


\bigskip
	
	\begin{thebibliography}{999999}
	\small
		
			\bibitem[AIS]{AIS} Joseph Ayoub, Florian Ivorra, Julien Sebag, \textit{Motives of rigid analytic tubes and nearby motivic sheaves}, Annales Scientifiques de l'Ecole Normale Superieure, 2017
			
			\bibitem[AJ]{AJ} Ran Azouri, Emil Jacobsen, \textit{The motivic Picard--Lefschetz formula}, arXiv:2510.12762, 2025
		
		\bibitem[Ay07]{Ay07a}Joseph Ayoub, \textit{Les six op\'erations de Grothendieck et le formalisme des cycles \'evanescents dans le monde motivique, I, II}. Ast\'erisque, volumes 314, 315 . Soci\'et\'e Math\'ematique de France, 2007 
		
		\bibitem[Ay14]{Ay14} Joseph Ayoub, \textit{La r\'ealisation \'etale et les op\'erations de Grothendieck}, Ann. Sci. Ec. Norm. Sup\'er. (4), 47(1):1–145, 2014
		
		\bibitem[Az]{Az} Ran Azouri, \textit{Motivic Euler characteristic of nearby cycles and a generalised quadratic conductor formula}, Journal of algebra, Volume 667, Pages 109-164, 2025
		
		\bibitem[BGV]{BGV} Federico Binda, Martin Gallauer, Alberto Vezzani, \textit{Motivic monodromy and p-adic cohomology theories.}, J. Eur. Math. Soc., 2025
				
		\bibitem[Blo]{Blo} Spencer Bloch, \textit{Cycles on arithmetic schemes and Euler characteristics of curves}, Algebraic geometry, Bowdoin, 1985, 421–450,
		Proc. Symp. Pure Math. 46, Part 2, Am. Math. Soc., Providence, RI, 1987
		
		\bibitem[BM]{BM} Dori Bejleri, Stephen McKean, \textit{Symmetric powers of null motivic Euler characteristic}, arXiv:2406.19506, 2024		
		
			\bibitem[BP]{BP} Dario Beraldo and Massimo Pippi, \textit{Proof of the Deligne-Milnor conjecture}, Preprint, arXiv:2410.02327, 2024
		
		\bibitem[BW]{BW} Tom Bachmann, Kirsten Wickelgren, \textit{$\Aa^1$-Euler classes: six functors formalisms, dualities, integrality and linear subspaces of complete intersections},	Journal of the Institute of Mathematics of Jussieu: pp. 1-66  , 2021
		
		\bibitem[DL98]{DL} Jean Denef, Francois Loeser, \textit{Motivic Igusa Zeta functions}, J. Algebraic Geom., 7, 505-537, 1998
		
		\bibitem[Hoy17]{Hoy} Marc Hoyois, \textit{The six operations in motivic homotopy theory}, Advances in Mathematics 305, 197-279, 2017
		
		\bibitem[Ill]{Ill} Luc Illusie, \textit{Sur la formule de Picard-Lefschetz}, in Algebraic geometry 2000, Azumino, Proceedings of the symposium, Nagano, Japan,
		July 20–30, 2000, pages 249–268, 2002
		
		\bibitem[KS]{KS} Kazuya Kato, Takeshi Saito, \textit{On the conductor formula of Bloch}, Publications Math\'ematiques de l'IHES, Tome 100, pp. 5-151, 2004
		
				\bibitem[KW]{KW} {Jesse Kass} and {Kirsten Wickelgren}, \textit{
			The class of Eisenbud-Khimshiashvili-Levine is the local $\mathbb{A}^1$-Brouwer degree}, Duke Math. J. 168, No. 3, 429--469, 2019
		
		
		\bibitem[LR]{LR} Marc Levine, Arpon Raksit, \textit{Motivic Gauss-Bonnet formulas},  Alg. Number Th. 14, 1801-1851, 2020
		
		\bibitem[Le20]{Le20a} Marc Levine, \textit{Aspects of enumerative geometry with quadratic forms},  Doc. Math. 25, 2179--2239, 2020
		
		\bibitem[Lef]{Lef} S. Lefschetz, \textit{L’analysis situs et la g\'eom\'etrie alg\'ebrique}, Collection
		de monographies sur la th\'eorie des fonctions, 1924
		
		\bibitem[LPLS]{Le20b} Marc Levine, Simon Pepin Lehalleur, Vasudevan Srinivas, \textit{Euler characteristics of homogeneous and weighted-homogeneous hypersurfaces}, Advances of Mathematics, Volume 441, April 2024
		
		\bibitem[May]{May} J. Peter May, \textit{The additivity of traces in triangulated categories}, Advances in Mathematics,
		163(1):34–73, 2001
	
		
		\bibitem[Mil]{Mil} John W. Milnor, \textit{Singular points of complex hypersurfaces}, Annals of Mathematics Studies,
		No. 61 Princeton University Press, Princeton, N.J.; University of Tokyo Press, Tokyo, 1968
		
		\bibitem[MO]{MO} John W. Milnor, Peter Orlik, \textit{Isolated singularities defined by weighted homogeneous polynomials} Topology 9, 385--393, 1970 
		
		
		\bibitem[Orm]{Orm} Kyle Ormsby, \textit{Milnor forms of algebraic singularities}, 
	https://kyleormsby.github.io/files/A1Milnor.pdf
	
	
		\bibitem[Pr]{Pr} Benedikt Preis, \textit{Motivic nearby cycles functors, local monodromy and universal local acyclicity}, arXiv:2305.03405
		, 2023
		
			\bibitem[PS97]{Pic} E. Picard and G. Simart, \textit{Th\'eorie des fonctions alg\'ebriques de deux variables ind\'ependantes}, Tome I, 1897
	
		\bibitem[PS14]{PS} Kate Ponto, Mike Shulman, \textit{Traces in symmetric monoidal categories}, Expositiones Mathematicae
		Volume 32, Issue 3, Pages 248-273, 2014
		
		\bibitem[Ro]{Ro} Markus Rost, \textit{The motive of a Pfister form}, https://www.math.uni-bielefeld.de/~rost/data/motive.pdf, 1998
		
		\bibitem[SGA7II]{SGA7II} Pierre Deligne and Nicholas Katz, \textit{Groupes de monodromie en g\'eom\'etrie alg\'ebrique}, (SGA 7 II), Lect. Notes Math. 340, Springer, Berlin-New York, 1973
	
		
		
	\end{thebibliography}
\end{document}